\documentclass{article}
\usepackage{graphicx}
\usepackage[utf8]{inputenc}
\usepackage{latexsym, amsmath, amsthm, amssymb, enumerate}
\newtheorem{cl}{Claim}[section]
\newtheorem{prop}{Proposition}[section]
\newtheorem{theorem}{Theorem}[section]
\newtheorem{definition}{Definition}[section]
\newtheorem{conjecture}{Conjecture}[section]
\newtheorem{corollary}{Corollary}[section]
\newtheorem{example}{Example}[section]
\newtheorem{lemma}{Lemma}[section]

\newtheorem*{lem:rectangles}{Lemma \ref{lem:rectangles}}
\usepackage[usenames, dvipsnames]{color}
\usepackage[title]{appendix}
\usepackage{caption}
\usepackage{graphicx}
\usepackage{amsmath}
\usepackage{mathtools}
\usepackage{breqn}
\usepackage{fancyhdr}
\usepackage{setspace,caption}
\usepackage{geometry}
\usepackage{lipsum}
\usepackage{makecell}
\usepackage{array}
\usepackage{indentfirst}
\usepackage[english]{babel}
\usepackage[symbol]{footmisc}
\usepackage{chngcntr}
\usepackage{tabu}
\usepackage{tcolorbox}
\usepackage{xhfill}
\counterwithin{table}{section}

\newcolumntype{C}[1]{>{\centering\arraybackslash}m{#1}}
\newcolumntype{N}{@{}m{0pt}@{}}
\makeatletter
\definecolor{light-gray}{gray}{0.7}
\definecolor{light-gray2}{gray}{0.85}

\newcommand{\thickhline}{%
    \noalign {\ifnum 0=`}\fi \hrule height 1pt
    \futurelet \reserved@a \@xhline
}
\newcolumntype{[}{@{\vrule width 1pt\hskip\tabcolsep}} \newcolumntype{]}{@{\hskip\tabcolsep\vrule width 1pt}}
\newcolumntype{"}{@{\hskip\tabcolsep\vrule width 1pt\hskip\tabcolsep}}
\makeatother
\begin{document}

\onehalfspacing

\begin{center}

{\Large\textbf{On a Conjecture by Hayashi on Finite Connected Quandles}} \\
{\textcolor{white}{.}} \\
{\Large{António Lages\footnote{Supported by FCT LisMath fellowship PD/BD/145348/2019.}\,\,\,\,\,\,\,\,\,\,Pedro Lopes\footnote{Supported by project FCT UIDB/04459/2020; member of CAMGSD.}}} \\
{\textcolor{white}{.}} \\
{\small{Department of Mathematics}} \\
{\small{Instituto Superior Técnico}} \\
{\small{Universidade de Lisboa}} \\
{\small{Av. Rovisco Pais, 1049-001 Lisbon, Portugal}} \\
{\small{\texttt{\{antonio.lages,pedro.f.lopes\}@tecnico.ulisboa.pt}}} \\
{\textcolor{white}{.}} \\

\end{center}

\pagenumbering{arabic}

\pagestyle{plain}

\begin{abstract}
A quandle is an algebraic structure whose binary operation is idempotent, right-invertible and right self-distributive. Right-invertibility ensures right translations are permutations and right self-distributivity ensures further they are automorphisms. For finite connected quandles, all right translations have the same cycle structure, called the profile of the connected quandle. Hayashi conjectured that the longest length in the profile of a finite connected quandle is a multiple of the remaining lengths. We prove that this conjecture is true for profiles with at most five lengths. \\
{\textcolor{white}{.}}

\end{abstract}

\section{Introduction}\label{sctn:intro}




We first define the algebraic structure known as \emph{quandle}, introduced independently in \cite{Joyce} and \cite{Matveev}.

\begin{definition}[Quandle]\label{def:quandle}
Let $X$ be a set equipped with a binary operation denoted by $*$. The pair $(X,*)$ is said to be a $\emph{quandle}$ if, for each $i, j, k\in X$,
\begin{enumerate}
\item $i*i=i$ (idempotency);
\item $\exists ! x\in X: x*j=i$ (right-invertibility);
\item $(i*j)*k=(i*k)*(j*k)$ (right self-distributivity).
\end{enumerate}
\end{definition}

To each quandle $(X,\ast)$ of order $n$ we associate a \emph{quandle table}. This is an $n\times n$ table whose element in row $i$ and column $j$ is $i\ast j$, for every $i,j\in X$. In the sequel, quandle tables will have an extra $0$-th column (where we display the $i$'s) and an extra $0$-th row (where we display the $j$'s) to improve legibility.

\begin{example}
Table \ref{table:1} is a quandle table for $Q_{9,4}$, a quandle of order $9$, see \cite{Vendramin}.

\begin{figure}[htbp]\centering
{\renewcommand{\arraystretch}{1.2}
\renewcommand{\tabcolsep}{6pt}
\begin{tabular}{|c|c c c c c c c c c|}
    \hline
    $*$ & 1 & 2 & 3 & 4 & 5 & 6 & 7 & 8 & 9\\
    \hline
    1 & 1 & 3 & 2 & 7 & 8 & 9 & 4 & 5 & 6\\
    2 & 3 & 2 & 1 & 9 & 6 & 5 & 8 & 7 & 4\\
    3 & 2 & 1 & 3 & 5 & 4 & 7 & 6 & 9 & 8\\
    4 & 5 & 7 & 9 & 4 & 1 & 8 & 2 & 6 & 3\\
    5 & 6 & 4 & 8 & 2 & 5 & 1 & 9 & 3 & 7\\
    6 & 7 & 9 & 5 & 8 & 3 & 6 & 1 & 4 & 2\\
    7 & 8 & 6 & 4 & 3 & 9 & 2 & 7 & 1 & 5\\
    8 & 9 & 5 & 7 & 6 & 2 & 4 & 3 & 8 & 1\\
    9 & 4 & 8 & 6 & 1 & 7 & 3 & 5 & 2 & 9\\
    \hline
\end{tabular}}
\captionof{table}{Quandle table for $Q_{9,4}$.}\label{table:1}
\end{figure}
\end{example}

For a quandle $(X,*)$ and for $i\in X$, we let $R_i: X\rightarrow X, j\mapsto j*i$ be the \emph{right translation} by $i$ in $(X,*)$ and $L_i: X\rightarrow X, j\mapsto i*j$ be the \emph{left translation} by $i$ in $(X,*)$. For example, the right translation by $1$ in $Q_{9,4}$ is $R_1=(1)(2\,\,3)(4\,\,5\,\,6\,\,7\,\,8\,\,9)$ and the left translation by $1$ in $Q_{9,4}$ is $L_1=(1)(2\,\,3)(4\,\,7)(5\,\,8)(6\,\,9)$. Axioms \textit{2.} and \textit{3.} in Definition \ref{def:quandle} ensure that each right translation in a quandle is an automorphism of that quandle.

The group $\langle R_i:i\in X\rangle$ generated by the right translations of a quandle $(X,*)$ is the \emph{right multiplication group}. A quandle $(X,*)$ is \emph{connected} if its right multiplication group acts transitively on $X$ and it is \emph{latin} if each of its left translations is a permutation of $X$. For instance, $Q_{9,4}$ is both connected and latin (by inspection of Table \ref{table:1}). In particular, a latin quandle is connected.

For a permutation $f$ on a finite set $X$, the \emph{cycle structure} of $f$ is the sequence $(1^{c_1}, 2^{c_2}, 3^{c_3},\dots)$, such that $c_m$ is the number of $m$-cycles in a decomposition of $f$ into disjoint cycles. For convenience, we omit entries with $c_i=0$ and we drop $c_i$ whenever $c_i=1$. For example, the right translation by $1$ in $Q_{9,4}$ has the cycle structure $(1, 2, 6)$ and the left translation by $1$ in $Q_{9,4}$ has the cycle structure $(1, 2^4)$.

For any mapping $f$ on a set $X=\{x_1,\dots, x_n\}$, consider the sequence $(\lvert f^{-1}(x_1)\rvert,\dots,\lvert f^{-1}(x_n)\rvert)$ and define the \emph{injectivity pattern} of $f$ as the previous sequence ordered in a nondecreasing fashion. For instance, the injectivity pattern of both the right and the left translation by $1$ in $Q_{9,4}$ is the tuple $(1,1,\dots,1)$.

\begin{lemma}\label{lem:connected}
In a connected quandle $(X,*)$, all right translations have the same cycle structure and all left translations have the same injectivity pattern.
\end{lemma}

\begin{proof}
Since the right multiplication group of $(X,*)$ acts transitively on $X$, for any $i, j \in X$, there exists $f\in \langle R_i:i\in X\rangle$ such that $i=f(j)$. Then, because  $f$ is an automorphism, $$R_i=R_{f(j)}=fR_jf^{-1} \qquad \text{ since, for any $k\in X$, }\qquad  R_{f(j)}(k)=k * f(j) = f(f^{-1}(k)*j)=fR_jf^{-1}(k) .$$ Analogously, $$L_i=L_{f(j)}=fL_jf^{-1} \qquad \text{ since, for any $k\in X$, }\qquad L_{f(j)}(k)=f(j) * k = f(j * f^{-1}(k))=fL_jf^{-1}(k) .$$ The result follows.
\end{proof}

Therefore, we define the \emph{profile} of a connected quandle $(X,*)$ to be the cycle structure of any of its right translations and the \emph{injectivity pattern} of a connected quandle $(X,*)$ to be the injectivity pattern of any of its left translations. For instance, $Q_{9,4}$ has profile $(1,2,6)$ and injectivity pattern $(1,1,\dots,1)$.

The following theorem  (\cite{Brieskorn}, \cite{FennRourke}) provides an equivalent description of the structure of a quandle in terms of its right translations. Our work is based upon this description of quandles.

\begin{theorem}\label{thm:equivdef}
Let $X$ be a set of order $n$ and suppose a permutation $R_i\in S_n$ is assigned to each $i\in X$. Then the expression $j*i:= R_i(j), \forall  j \in X$, yields a quandle structure on $X$ if and only if $R_{ R_i(j)}=  R_i R_j R_i^{-1}$ and $R_i(i)=i$, $\forall i, j\in X$. This quandle structure is uniquely determined by the set of $n$ permutations.
\end{theorem}

\begin{proof}
The proof is straightforward. Let $i,j\in X$. Each $R_i$ is a permutation due to right-invertibility; idempotency implies $R_i(i)=i$; right self-distributivity implies $R_{ R_i(j)}=  R_i R_j R_i^{-1}$. The details can be found in \cite{Brieskorn}.
\end{proof}

In general, the profile of a finite connected quandle is of the form $(\ell_1,\dots,\ell_c)$, for a certain $c\in\mathbb{Z}^+$, where $1=\ell_1\leq\dots\leq \ell_c$. This is the format for the profile of a quandle that we will use below, allowing for repeats. Note that the profile of a finite connected quandle $(X,*)$ has to have at least one $1$, since $R_i(i)=i$, for each $i\in X$.

\begin{theorem}[cf. \cite{Lages_Lopes-latin}]\label{thm:latin}
Let $(X,*)$ be a quandle of order $n$ with profile $(\ell_1,\dots, \ell_c)$, where $1=\ell_1<\cdots<\ell_c$. Then $(X,*)$ is a latin quandle.
\end{theorem}

\begin{proof}
See \cite{Lages_Lopes-latin}.
\end{proof}

We now present  Hayashi's Conjecture on the profile of finite connected quandles (\cite{Hayashi}).

\begin{conjecture}[Hayashi's Conjecture]\label{cnj:Hayashi}
Let $(X,*)$ be a connected quandle of order $n$ with profile $(\ell_1,\dots,\ell_c)$, where $1=\ell_1\leq\cdots\leq \ell_c$. Then $\ell_c$ is a multiple of $\ell_i$ for every $i\in\{1,\dots,c\}$.
\end{conjecture}

Hayashi's Conjecture is trivial for $c=2$. Quandles with $c=2$ are called quandles of cyclic type (\cite{Lopes_Roseman}). Watanabe (\cite{Watanabe}) verified Hayashi's Conjecture both for $c=3$ and $n\leq47$. We check it for $c\in\{3,4,5\}$  using the description of quandles in terms of their right translations. Specifically, we prove the Main Theorem (Theorem \ref{thm:theone}).

\begin{theorem}[Main Theorem]\label{thm:theone}
Hayashi's Conjecture is true for $c\in\{3,4,5\}$.
\end{theorem}


The article is organized as follows. In Section \ref{sctn:aux} we elaborate further on quandles to pave the way to the proof of the Main Theorem. In Section \ref{sctn:hayashi}, we prove the Main Theorem. In Section \ref{sctn:further} we make some final remarks.
\section{Finite connected quandles}\label{sctn:aux}

In this section we elaborate further on finite connected quandles. Theorem \ref{thm:generalization} generalizes an obstruction on the profile of finite connected quandles due to Rehman (\cite{Rehman}). Lemmas \ref{lem:reh1}, \ref{lem:reh2}, \ref{lem:reh3} can be found in \cite{Rehman}. We state and prove them here for the reader's convenience.

\begin{lemma}\label{lem:reh1}
Let $(X,*)$ be a finite connected quandle, let $Y\subsetneq X$ be a subquandle of $X$ and let $Y^c:=X\setminus Y$. Then $X=\{((\cdots(y_n*y_{n-1})*\cdots)*y_2)*y_1\mid y_1,\dots,y_n\in Y^c, n\geq1\}$
\end{lemma}

\begin{proof}
For $W, W'\subseteq X$, we let $W*W'=\{ w*w'\,|\, w\in W, w'\in W'\}$.
First, we note that $Y*Y\subseteq Y$ (since $Y$ is a subquandle) so we have that $Y^c*Y\subseteq Y^c$. Let $Z:=\{((\cdots(y_n*y_{n-1})*\cdots)*y_2)*y_1\mid y_1,\dots,y_n\in Y^c, n\geq1\}$. We prove that $Z*X\subseteq Z$. On one hand, $Z*Y^c\subseteq Z$ by definition. On the other hand, given $y\in Y$ and $((\cdots(y_n*y_{n-1})*\cdots)*y_2)*y_1\in Z$, we have that \[((\cdots(y_n*y_{n-1})*\cdots*y_2)*y_1)*y=((\cdots((y_n*y)*(y_{n-1}*y))*\cdots)*(y_2*y))*(y_1*y).\] As $Y^c*Y\subseteq Y^c$, we have that $Z*Y\subseteq Z$, so $Z*X\subseteq Z$. Since $X$ is connected, we conclude that $X=Z$.
\end{proof}

\begin{lemma}\label{lem:reh2}
Let $(X,*)$ be a finite connected quandle and let $Y,Z\subseteq X$ be two subquandles of $X$ such that $X=Y\cup Z$. Then either $X=Y$ or $X=Z$.
\end{lemma}

\begin{proof}
If $X\neq Y$, then, by Lemma \ref{lem:reh1}, $X=\{((\cdots(y_n*y_{n-1})*\cdots)*y_2)*y_1\mid n\geq1,y_1,\dots,y_n\in Y^c\}$, where $Y^c=X\setminus Y$. Since $Z$ is a subquandle of $X$ such that $Z\supseteq Y^c$, we conclude that $X=Z$.
\end{proof}


\begin{lemma}\label{lem:reh3}
Let $(X,*)$ be a finite connected quandle, let $x\in X$ and $p\in\mathbb{Z}$. Then $Y=\{y\in X\mid R_x^p(y)=y\}$ is a subquandle of $X$.
\end{lemma}

\begin{proof}
Noting that composition of automorphisms is an automorphism (we will use this remark again without mentioning it) $R_x^p(y*y')=R_x^p(y)*R_x^p(y')=y*y'$, for every $y,y'\in Y$. Then, $Y=\{y\in X\mid R_x^p(y)=y\}$ is  closed under the $*$ operation. Moreover, let $a,b\in Y$. Then, there exists a unique $u\in X$ such that $u*b=a$. Then, $$a=R_x^p(a)=R_x^p(u*b)=R_x^p(u)*R_x^p(b)=R_x^p(u)*b ,$$which implies $u=R_x^p(u)$ by uniqueness and therefore right-invertibility holds over $Y$. $Y$ is a subquandle of $X$.
\end{proof}

\begin{definition}
Let $L$ be a finite sequence of elements from $\mathbf{Z}^+$. We let
\[
\mathrm{lcm}\,\, L = \begin{cases} 1, & \text{if $L=\emptyset$}\\
                             \text{least common multiple of elements in $L$}, & \text{otherwise}
                \end{cases}
\]
\end{definition}

Theorem \ref{thm:generalization} is a generalization of Proposition 3.5 in \cite{Rehman}.

\begin{theorem}\label{thm:generalization}
Let $(X,*)$ be a connected quandle with profile $(\ell_1,\dots,\ell_c)$, let $L=\{\ell_i\mid 1\leq i\leq c\}$, let $P=\{p_1,\dots,p_j\}\subseteq L$ and $Q=\{q_1,\dots,q_k\}\subseteq L$  such that $P\cup Q=L$ and let $p=\mathrm{lcm}(p_1,\dots,p_j)$ and $q=\mathrm{lcm}(q_1,\dots,q_k)$. Then either $p\mid q$ or $q\mid p$.
\end{theorem}

\begin{proof}
Given $x\in X$, let $X_p=\{y\in X\mid R_x^p(y)=y\}$ and $X_q=\{z\in X\mid R_x^q(z)=z\}$. We know both $X_p$ and $X_q$ are subquandles of $X$ by Lemma \ref{lem:reh3}. Moreover, $X=X_p\cup X_q$. Then, by Lemma \ref{lem:reh2}, either $X=X_p$ or $X=X_q$, which means that either $q\mid p$ or $p\mid q$.
\end{proof}

\begin{corollary}\label{cor:quasihayashi}
Let $(X,*)$ be a connected quandle with profile $(\ell_1,\dots,\ell_c)$ and let $\ell=\mathrm{lcm}(\ell_i\, : \, \ell_i\nmid \ell_c)$. Then either $\ell\mid\ell_c$ or $\ell_c\mid\ell$.
\end{corollary}

\begin{proof}
Take $P=\{\ell_i\, : \, \ell_i\mid \ell_c\}$ and $Q=\{\ell_i\, : \, \ell_i\nmid \ell_c\}$ and apply  Theorem \ref{thm:generalization}. The result follows.
\end{proof}

Corollary \ref{cor:quasihayashi} states that given a connected quandle $(X,*)$ with profile $(\ell_1,\dots,\ell_c)$ either $(X,*)$ satisfies Hayashi's conjecture or $\ell_c\mid\ell$, where $\ell=\mathrm{lcm}(\ell_i\, : \, \ell_i\nmid \ell_c)$. There is no other alternative.

We now introduce notation and a theorem, Theorem \ref{thm:perm}, concerning how the right translations relate to one another.

\begin{definition}\label{def:quandleNotation}
Let $(X,*)$ be a connected quandle of order $n$ with profile  $(\ell_1,\dots, \ell_c)$, where $1=\ell_1\leq\cdots\leq \ell_c$. We set \vspace{-5px}\begin{align*}
a_0&:=0;\\
a_s&:=\sum_{r=1}^{s}\ell_r,\quad\text{for   }s\in\{1,\dots,c\};\\
a_s'&:=a_{s-1}+1,\quad\text{for   }s\in\{1,\dots,c\};\\
C_s&:=\{a_s',\dots,a_s\}\subseteq X,\quad\text{for   }s\in\{1,\dots,c\};\\
C_{s_1,\dots,s_r}&:=\bigcup\nolimits_{s\in\{s_1,\dots,s_r\}}C_s,\quad\text{for   }r\in\{1,\dots,c\}\text{   and   }1\leq s_1<\cdots<s_r\leq c.
\end{align*}
Note that $a_1=a_1'=1$, $a_2'=2$ and $a_c=n$. Besides, $\ell_s= |C_s|,\text{ for   }s\in\{1,\dots,c\}$, $C_1=\{1\}$ and the $C_s$'s form a partition of $X=\{1,\dots,n\}$.
\end{definition}

For example, for the quandle $Q_{9,4}$ in Table \ref{table:1}, $C_1=\{1\}$, $C_2=\{2,3\}$ and $C_3=\{4,5,6,7,8,9\}$ constitute a partition of $X=\{1,\dots,9\}$. \textbf{Unless otherwise stated in the sequel, $(X,*)$ will be as in Definition \ref{def:quandleNotation}. We will keep to the notation of Definition \ref{def:quandleNotation}.}\par

\begin{theorem}\label{thm:perm}
Let $(X,*)$ be a connected quandle of order $n$ with profile  $(\ell_1,\dots,\ell_c)$, where $1=\ell_1\leq\cdots\leq \ell_c$. Then, modulo renaming the elements, the right translations of $(X,*)$ satisfy the following conditions:

\begin{enumerate}
    \item $R_1=(1)(2\,\cdots\,a_{2})(a_{3}'\,\cdots\,a_{3})\,\cdots\,(a_{c}'\,\cdots\,n)$;
    \item $R_{a_{s-1}+k}=R_1^{k}R_{a_s}R_1^{-k}$, for $s\in\{1,\dots,c\}$ and for $k\in\{1,\dots,\ell_s\}$;
    \item $R_1^{R_{a_s}(1)-a_{t-1}}R_{a_t}R_1^{-(R_{a_s}(1)-a_{t-1})}=R_{a_s}R_1 R_{a_s}^{-1}$, for $s,t\in\{1,\dots,c\}$ such that $R_{a_s}(1)\in C_t$;
    \item if $R_j(i)=i$, then $R_i R_j=R_j R_i$, for $i,j\in X$;
    \item $R_{a_t}^{-1} R_1 R_{a_t}= R_1^k R_{a_s} R_1^{-k}$, for $s,t\in\{1,\dots,c\}$ where it has been assumed that $R_{a_t}(a_{s-1}+k)=1$,  for some $k\in\{1,\dots,\ell_s\}$.
\end{enumerate}
\end{theorem}

\begin{proof}
We prove assertions \textit{1.} to \textit{5.}.\par

\vspace{5px}

\textit{1.} We may assume that $R_1=(1)(2\,\cdots\,a_{2})(a_{3}'\,\cdots\,a_{3})\,\cdots\,(a_{c}'\,\cdots\,n)$ without loss of generality. If necessary, we might relabel the indices. \textbf{This expression for $R_1$ will be assumed in the sequel.} Notice that the elements of $(X,*)$ that belong to the cycle $(a_s',\dots,a_s)$ are precisely the elements of $C_s$, for each $s\in\{1,\dots,c\}$. Moreover, the length of the cycle $(a_s',\dots,a_s)$ is $\ell_s$, for each $s\in\{1,\dots,c\}$.\par

\vspace{5px}

\textit{2.} Let $s\in\{1,\dots,c\}$. First, we have $R_{a_{s-1}+1}=R_{R_1(a_s)}=R_1 R_{a_s}R_1^{-1}$ by assertion \textit{1.} and Theorem \ref{thm:equivdef}. Now, if $R_{a_{s-1}+k}= R_1^k R_{a_s} R_1^{-k}$, for $k\in\{1,\dots,\ell_s-1\}$, again by assertion \textit{1.} and Theorem \ref{thm:equivdef}, we get \[ R_{a_{s-1}+k+1}= R_{ R_1(a_{s-1}+k)}= R_1 R_{a_{s-1}+k} R_1^{-1}= R_1 R_1^{k} R_{a_s} R_1^{-k} R_1^{-1}= R_1^{k+1} R_{a_s} R_1^{-(k+1)}.\] Hence, we conclude, by induction, that $R_{a_{s-1}+k}= R_1^k R_{a_s} R_1^{-k}$, for $s\in\{1,\dots,c\}$ and $k\in\{1,\dots,\ell_s\}$.\par

\vspace{5px}

\textit{3.} Let $s,t\in\{1,\dots,c\}$ be such that $R_{a_s}(1)\in C_t$. By assertion \textit{2.}, we have that \[R_{ R_{a_s}(1)}=R_{a_{t-1}+(R_{a_s}(1)-a_{t-1})}=R_1^{R_{a_s}(1)-a_{t-1}}R_{a_t}R_1^{-( R_{a_s}(1)-a_{t-1})}.\] Since $R_{R_{a_s}(1)}=R_{a_s}R_1 R_{a_s}^{-1}$ by Theorem \ref{thm:equivdef}, we conclude that $R_1^{R_{a_s}(1)-a_{t-1}}R_{a_t}R_1^{-(R_{a_s}(1)-a_{t-1})}=R_{a_s}R_1 R_{a_s}^{-1}$, for $s,t\in\{1,\dots,c\}$ such that $R_{a_s}(1)\in C_t$.

\vspace{5px}

\textit{4.} Let $i,j\in X$ be such that $R_j(i)=i$. Then by Theorem \ref{thm:equivdef}, $R_i=R_{R_j(i)}=R_j R_i R_j^{-1}\Leftrightarrow R_i R_j=R_j R_i$.

\vspace{5px}

\textit{5.} Let $s,t\in\{1,\dots,c\}$ and $k\in\{1,\dots,\ell_s\}$ satisfy $R_{a_t}(a_{s-1}+k)=1$. By Theorem \ref{thm:equivdef} and assertion \textit{2.}, \[ R_1= R_{ R_{a_t}(a_{s-1}+k)}= R_{a_t} R_{a_{s-1}+k} R_{a_t}^{-1}= R_{a_t} R_1^k R_{a_s} R_1^{-k} R_{a_t}^{-1}\Leftrightarrow R_{a_t}^{-1} R_1 R_{a_t}= R_1^k R_{a_s} R_1^{-k},\] where $s,t\in\{1,\dots,c\}$ and $k\in\{1,\dots,\ell_s\}$ are such that $R_{a_t}(a_{s-1}+k)=1$.
\end{proof}

Because least common multiples of pairs of integers are going to be used as exponents in the sequel, we resort to a lighter notation.
\begin{definition}
Given  $i,j\in\mathbb{Z}^+$, we let $[i,j]$ stand for their least common multiple.
\end{definition}

\begin{prop}\label{prop:main}
Let $(X, *)$ be a finite quandle. Let $t,u \in\{1,\dots,c\}$. Let $i_t\in C_t$, $i_u\in C_u$ with $i_t*i_u\in C_v$, for some $v \in\{1,\dots,c\}$. Then $\ell_v\mid[\ell_t,\ell_u]$.
\end{prop}

\begin{proof}
As $[\ell_t,\ell_u]$ is a multiple of $\ell_t$ and $i_t\in C_t$, then $R_1^{[\ell_t,\ell_u]}(i_t) = i_t$. Analogously, $R_1^{[\ell_t,\ell_u]}(i_u) = i_u$. Then, \[i_t*i_u=  R_1^{[\ell_t,\ell_u]}(i_t)* R_1^{[\ell_t,\ell_u]}(i_u) =  R_1^{[\ell_t,\ell_u]}(i_t*i_u).\] Note that $( R_1^{\,m} (i_t*i_u)\, :\, m\in \mathbb{Z}_0^+)$ is a periodic sequence, whose period is $\ell_v$. The sequence starts in $i_t*i_u$, and is formed by copies of the cycle $C_v$. Then $i_t*i_u= R_1^{[\ell_t,\ell_u]}(i_t*i_u)$ implies $\ell_v\mid[\ell_t,\ell_u]$. This completes the proof.
\end{proof}

\begin{definition}
For $t,u\in\{1,\dots,c\}$ set $$\mathcal{R}_{t,u}:=\{x*y:x\in C_t\wedge y\in C_u\} .$$ Note that each $\mathcal{R}_{t,u}$ is non-empty.
\end{definition}


For example, let $(X,*)=Q_{9,4}$, cf. Table \ref{table:1}. Then, $\mathcal{R}_{3,2}=\{4,5,6,7,8,9\}$ while $\mathcal{R}_{2,1}=\{2,3\}$.

\begin{corollary}\label{cor:main}
Let $t,u\in\{1,\dots,c\}$ and consider $\mathcal{R}_{t,u}$. Let $I:=\{w:\ell_w\mid[\ell_t,\ell_u]\}$. Then $\mathcal{R}_{t,u}\subseteq\bigcup_{x\in I}C_x$.
\end{corollary}

\begin{proof}
Let $t,u \in\{1,\dots,c\}$. Let $i_t\in C_t$, $i_u\in C_u$ with $i_t*i_u\in C_v$, for some $v \in\{1,\dots,c\}$.
Then, by Proposition \ref{prop:main}, $\ell_v\mid[\ell_t,\ell_u]$, so $v\in I$ and $i_t*i_u\in\bigcup_{x\in I}C_x$. In particular, $\mathcal{R}_{t,u}\subseteq\bigcup_{x\in I}C_x$.
\end{proof}

\begin{corollary}\label{cor:content}
Let $t,u\in\{1,\dots,c\}$ such that $\ell_t\nmid \ell_u$, and consider $\mathcal{R}_{t,u}$. Let $I:=\{w:\ell_w\nmid \ell_u\wedge(\ell_u\nmid \ell_w\vee \ell_t\mid \ell_w)\}$. Then $\mathcal{R}_{t,u}\subseteq\bigcup_{x\in I}C_x$.
\end{corollary}

\begin{proof}
Let $t,u \in\{1,\dots,c\}$. Let $i_t\in C_t$, $i_u\in C_u$ with $i_t*i_u\in C_v$, for some $v \in\{1,\dots,c\}$, and such that $\ell_v\mid \ell_u\vee(\ell_u\mid \ell_v\wedge \ell_t\nmid \ell_v)$. Therefore $[\ell_u,\ell_v]\in\{\ell_u,\ell_v\}$. Also, $R_1^{[\ell_u,\ell_v]}(i_t)* R_1^{[\ell_u,\ell_v]}(i_u)= R_1^{[\ell_u,\ell_v]}(i_t*i_u)\Leftrightarrow  R_1^{[\ell_u,\ell_v]}(i_t)*i_u=i_t*i_u$. However, $i_t\neq  R_1^{[\ell_u,\ell_v]}(i_t)$ whether $[\ell_u,\ell_v]=\ell_u$ (since $\ell_t\nmid \ell_u$) or $[\ell_u,\ell_v]=\ell_v$ (since, in this case, $\ell_u\mid \ell_v\wedge \ell_t\nmid \ell_v$). Hence $R_{i_u}$ is not injective, which is a contradiction. Therefore $v\in I$ and $i_t*i_u\in\bigcup_{x\in I}C_x$. In particular, $\mathcal{R}_{t,u}\subseteq\bigcup_{x\in I}C_x$.
\end{proof}

\begin{corollary}\label{cor:latin}
Suppose the conditions of Corollary \ref{cor:content} are
satisfied.  Assume further that $(X,*)$ is latin and $\ell_u\nmid \ell_t$. Let $J:=\{w:\ell_w\nmid \ell_u\wedge \ell_w\nmid \ell_t \wedge(\ell_u\nmid \ell_w \vee \ell_t\mid \ell_w) \wedge (\ell_t\nmid \ell_w \vee \ell_u\mid \ell_w)\}$. Then $\mathcal{R}_{t,u}\subseteq\bigcup_{x\in J}C_x$.
\end{corollary}

\begin{proof}
The original part of this proof mimics the proof of Corollary \ref{cor:content} but applied to a left translation.
Let $t,u \in\{1,\dots,c\}$. Let $i_t\in C_t$, $i_u\in C_u$ with $i_t*i_u\in C_v$, for some $v \in\{1,\dots,c\}$, and such that $\ell_v\mid \ell_t\vee(\ell_t\mid \ell_v\wedge \ell_u\nmid \ell_v)$. Therefore $[\ell_t,\ell_v]\in\{\ell_t,\ell_v\}$. Furthermore,  $R_1^{[\ell_t,\ell_v]}(i_t)* R_1^{[\ell_t,\ell_v]}(i_u)= R_1^{[\ell_t,\ell_v]}(i_t*i_u)\Leftrightarrow i_t*R_1^{[\ell_t,\ell_v]}(i_u)=i_t*i_u$. However, $i_u\neq R_1^{[\ell_t,\ell_v]}(i_u)$, whether $[\ell_t,\ell_v]=\ell_t$ or $[\ell_t,\ell_v]=\ell_v$. So $L_{i_t}$ is not injective, hence $(X,*)$ is not latin, which is a contradiction. Whence $v\in I'=\{w:\ell_w\nmid \ell_t \wedge(\ell_t\nmid \ell_w \vee \ell_u\mid \ell_w)\}$ and $i_t*i_u\in\bigcup_{x\in I'}C_x$. In particular, $\mathcal{R}_{t,u}\subseteq\bigcup_{x\in I'}C_x$. As Corollary \ref{cor:content} also applies, we get $\mathcal{R}_{t,u}\subseteq \bigcup_{x\in J}C_x$.
\end{proof}

\begin{corollary}\label{cor:byone}
Let $t, u\in\{1,\dots,c\}$ such that $ C_t =\{ i_t \}$ $($for $i_t\in X)$ and consider $\mathcal{R}_{t,u}$. Then $\mathcal{R}_{t,u}\subseteq C_v$, for $v\in\{1,\dots,c\}$ such that $l_v\mid l_u$. Furthermore, for each $i_v\in C_v$, there are $l_u/l_v$ solutions over $C_u$ for the equation $i_t*x=i_v$.
\end{corollary}

\begin{proof}
Assume $C_t=\{ i_t \}$ and let $i_u\in C_u$ and $i_v\in C_v$ such that $i_t*i_u=i_v$. Then, by Corollary \ref{cor:main}, $l_v\mid l_u$,
\[
R_1^{m}(i_t)*R_1^{m}(i_u)= R_1^{m}(i_v)\quad  \Leftrightarrow\quad   i_t\ast \underset{\text{ mod } l_u}{\text{(}i_u+m\text{)}}= \underset{\text{ mod } l_v}{\text{(}i_v+m\text{)}}\,\, ,
\]
for each $m\in\{1,\dots,\ell_u\}$. Then, with $m\in \{ 1, 2, \dots , l_v\}$ and $k\in \{ 1, 2, \dots , l_u/l_v \}$
\[
i_t* (i_u + m + k l_v) = i_v+m+k l_v \underset{\text{ mod } l_v}{=} i_v+m = i_t* (i_u+m)\,\, .
\]
Then, for each $i_v\in C_c$, there are $l_u/l_v$ distinct $x\in C_u$ such that $i_t * x = i_v$.


\end{proof}

\begin{definition}
Let $(X,*)$ be a connected quandle with profile $(\ell_1,\dots,\ell_c)$, where $1=\ell_1\leq\dots\leq \ell_c$. A \emph{cycle quandle table} for $(X,*)$ is a $c\times c$ table whose element in row $t$ and column $u$ is any subset $C\subseteq X$ satisfying $\mathcal{R}_{t,u}:=C_t*C_u\subseteq C$. For convenience, we omit $C$ whenever our best guess is $C=X$ or it is not relevant for the discussion at issue. In the sequel, cycle quandle tables will have an extra $0$-th column (where we display the $C_t$'s) and an extra $0$-th row (where we display the $C_u$'s) to improve legibility.
\end{definition}

\begin{example}
Table \ref{table:2} is an example of a cycle quandle table for $Q_{9,4}$ (cf. Table \ref{table:1}).

\begin{figure}[htbp]\centering
{\renewcommand{\arraystretch}{1.5}
\begin{tabular}{[C{24pt}"C{24pt}|C{24pt}|C{24pt}]}
    \thickhline
    $*$ & $C_1$ & $C_2$ & $C_3$\\
    \thickhline
    $C_1$ & $C_1$ & $C_2$ & $C_3$\\
    \hline
    $C_2$ & $C_2$ & $C_{1,2}$ & $C_3$\\
    \hline
    $C_3$ & $C_3$ & $C_3$ & \\
    \thickhline
\end{tabular}}
\captionof{table}{A cycle quandle table for $Q_{9,4}$.}\label{table:2}
\end{figure}
\end{example}

\begin{prop}\label{prop:2equals}
Let $(X, *)$ be a finite connected quandle whose profile is $(\ell_1, \dots \, \ell_c)$ with $$1=\ell_1<\ell_2< \cdots <\ell_i=\ell_{i+1}< \cdots <\ell_c$$ $$\text{where } \qquad \qquad 2\leq i\leq c-1 \qquad \qquad \text{ and }\qquad \qquad \ell_j\nmid \ell_k \quad \text{ for }\quad j,k\in \{2, \dots , c\}\setminus \{i+1\} \quad \text{ with }\quad  j\neq k .$$Then, the largest term in the injectivity pattern of $(X, *)$ is, at most, $2$.
\end{prop}
\begin{proof}
We keep the notation of the statement. For $k\notin \{ i, i+1\}$ $${\cal R}_{1, k}\subseteq C_k \qquad {\cal R}_{1, i}\subseteq C_{i, i+1} \qquad {\cal R}_{1, i+1}\subseteq C_{i, i+1}$$ since $[\ell_1, \ell_k]=\ell_k$ and the only $\ell_i$'s that divide $\ell_k$ are $\ell_1(=1)$ and $\ell_k$, using Corollary \ref{cor:main} and further noting that $R_k$ only has one fixed point. Thus, for $k\notin \{ i, i+1\}$, the injectivity pattern of $L_1$ over $C_k$ is $\ell_k/\ell_k=1$, according to Corollary \ref{cor:byone}. For $k\in \{i, i+1\}$, there could be $d_i, c_i\in C_i$ and $c_{i+1}\in C_{i+1}$ such that $$L_1(c_i)=d_i=L_1(c_{i+1}) \quad \Longrightarrow \quad L_1^{-1}(\{d_i\})=\{c_i, c_{i+1}\} .$$Note that there cannot be more than $2$ pre-images due to Corollary \ref{cor:byone} and $\ell_i/\ell_i=1=\ell_{i+1}/\ell_{i+1}$. Since $(X, *)$ is connected by hypothesis, the result follows from Lemma \ref{lem:connected}.
\end{proof}

\begin{corollary}\label{cor:5with2equals}
We keep the conditions of Proposition \ref{prop:2equals}. There is no such quandle when $c=5$.
\end{corollary}
\begin{proof}
We consider the cases:
\begin{enumerate}[1.]
  \item \qquad $1=\ell_1<\ell_2=\ell_3<\ell_4 <\ell_5$
  \item \qquad $1=\ell_1<\ell_2<\ell_3=\ell_4 <\ell_5 \quad (\text{respect.,}\quad  1=\ell_1<\ell_2<\ell_3<\ell_4 =\ell_5)$
\end{enumerate}
Here are the proofs for each of them.
\begin{enumerate}[1.]
  \item \qquad $1=\ell_1<\ell_2=\ell_3<\ell_4 <\ell_5$:
$\mathcal{R}_{1,5}\subseteq C_5$ (arguing as in the proof of Proposition \ref{prop:2equals}) and $\mathcal{R}_{5,5}\subseteq C_{1, 5}$ (since $\ell_i\nmid \ell_5$ for $i\in \{ 2, 3, 4\}$ and using Corollary \ref{cor:main}$)$.
We check that $\mathcal{R}_{2,5},\mathcal{R}_{3,5}\subseteq C_4$. First, by Corollary \ref{cor:content}, $\mathcal{R}_{2,5},\mathcal{R}_{3,5}\subseteq C_{2,3,4}$. Proof for the $\mathcal{R}_{2,5}\subseteq C_{2,3,4}$ case (the other one is analogous): with $i\in \{2,3,4\}$, since $\ell_i\nmid \ell_5$ and $\ell_5\nmid \ell_i$ along with $\ell_1, \ell_5\mid \ell_5$, the result follows from Corollary \ref{cor:content}. Now, suppose $i\in C_{2,3}$ and $j\in C_5$ are such that $i*j=k\in C_{2,3}$. Then, $R_1^{\ell_3}(i)*R_1^{\ell_3}(j)=R_1^{\ell_3}(k)\Leftrightarrow i* R_1^{\ell_3}(j)=k$ and $R_1^{\ell_3}(j)\neq j$ because $\ell_5\nmid \ell_3$. Also, since $R_1^{2\ell_3}(i)*R_1^{2\ell_3}(j)= R_1^{2\ell_3}(k)\Leftrightarrow i* R_1^{2\ell_3}(j)=k$. If $\ell_5\mid 2\ell_3$ then $\ell\ell_5=2\ell_3$ with $\ell=1$ since $\ell_3<\ell_5$. But then $\ell_3\mid 2\ell_3=\ell_5$ which conflicts with the standing assumptions. So $\ell_5\nmid 2\ell_3$ and $R_1^{2\ell_3}(j)\neq j$. Then $j\neq R_1^{\ell_3}(j)\neq R_1^{2\ell_3}(j)\neq j$ and  there is, at least, a $3$ in the injectivity pattern of $L_i$, which  conflicts with Proposition \ref{prop:2equals}. Thus, $\mathcal{R}_{2,5},\mathcal{R}_{3,5}\subseteq C_4$.

Now we check that $\mathcal{R}_{4,5}\subseteq C_{2,3}$, leaving the details for the reader since they are analogous to the ones in the preceding paragraph. First, by Corollary \ref{cor:content}, $\mathcal{R}_{4,5}\subseteq C_{2,3,4}$. Now, suppose $i\in C_4$ and $j\in C_5$ are such that $i*j=k\in C_4$. Therefore,  $R_1^{\ell_4}(i)*R_1^{\ell_4}(j)=R_1^{\ell_4}(k)\Leftrightarrow i* R_1^{\ell_4}(j)=k$ and also $R_1^{2\ell_4}(i)* R_1^{2\ell_4}(j)= R_1^{2\ell_4}(k)\Leftrightarrow i* R_1^{2\ell_4}(j)=k$. But as $j\neq  R_1^{\ell_4}(j)\neq R_1^{2\ell_4}(j)\neq j$, there is, at least, a $3$ in the injectivity pattern of $L_i$, which conflicts with Proposition \ref{prop:2equals}. So, $\mathcal{R}_{4,5}\subseteq C_{2,3}$ and Table \ref{table:5} is a cycle quandle table for the quandle at issue.

\begin{figure}[htbp]\centering
{\renewcommand{\arraystretch}{1.5}
\begin{tabular}{[C{24pt}"C{24pt}|C{24pt}|C{24pt}|C{24pt}|C{24pt}]}
    \thickhline
    $*$ & $C_1$ & $C_2$ & $C_3$ & $C_4$ & $C_5$\\
    \thickhline
    $C_1$ &  &  &  &  & $C_5$\\
    \hline
    $C_2$ &  &  &  &  & $C_4$\\
    \hline
    $C_3$ &  &  &  &  & $C_4$\\
    \hline
    $C_4$ &  &  &  &  & $C_{2,3}$\\
    \hline
    $C_5$ &  &  &  &  & $C_{1,5}$\\
    \thickhline
\end{tabular}}
\captionof{table}{Cycle quandle table for $(X,*)$ when $1=\ell_1<\ell_2=\ell_3<\ell_4<\ell_5$ and $\ell_3\nmid \ell_4$, $\ell_3\nmid \ell_5$, $\ell_4\nmid \ell_5$.}\label{table:5}
\end{figure}

Inspecting the last column of Table \ref{table:5}, we see that, by right translation, $C_5$ maps $C_4$ into $C_{2, 3}$. Since there are no more occurrences of $C_2$ or $C_3$ along that column, then $\lvert C_{2,3}\rvert=\lvert C_4\rvert$. Therefore, $\ell_3\mid \ell_4$, which conflicts with the standing assumptions.

  \item \qquad $1=\ell_1<\ell_2<\ell_3=\ell_4 <\ell_5 \quad (\text{respect.,}\quad  1=\ell_1<\ell_2<\ell_3<\ell_4=\ell_5)$:

We check that $\mathcal{R}_{3,5}\subseteq C_2$. First, by Corollary \ref{cor:content}, $\mathcal{R}_{3,5}\subseteq C_{2,3,4}$ (respect., $\mathcal{R}_{3,5}\subseteq C_{2,3}$). Now, suppose $i\in C_3$ and $j\in C_5$ are such that $i*j=k\in C_{3,4}$ (respect., $i*j=k\in C_3$). Then,  $R_1^{\ell_3}(i)*R_1^{\ell_3}(j)= R_1^{\ell_3}(k)\Leftrightarrow i* R_1^{\ell_3}(j)=k$ and $R_1^{2\ell_3}(i)* R_1^{2\ell_3}(j)= R_1^{2\ell_3}(k)\Leftrightarrow i* R_1^{2\ell_3}(j)=k$. But as $j\neq  R_1^{\ell_3}(j)\neq R_1^{2\ell_3}(j)\neq j$, there is, at least, a $3$ in the injectivity pattern of $L_i$, which conflicts with Proposition \ref{prop:2equals}. Thus, $\mathcal{R}_{3,5}\subseteq C_2$. But since $\{x*5\,|\, x\in C_3\}\subseteq \mathcal{R}_{3,5}\subseteq C_2$, this implies that $|C_3|\leq |C_2|$ which conflicts with the standing assumptions.
\end{enumerate}
Thus, with $c=5$ there is no quandle that satisfies the conditions of  Proposition \ref{prop:2equals} and the proof of Corollary \ref{cor:5with2equals} is complete.
\end{proof}

\section{Hayashi's conjecture is true for $c\in \{ 3, 4, 5 \}$}\label{sctn:hayashi}

In this section we prove Theorem \ref{thm:theone}. Specifically, we prove Hayashi's Conjecture for $c=3$, $c=4$, $c=5$ in Subsections \ref{subsctn:3}, \ref{subsctn:4}, \ref{subsctn:5}, respectively. The proofs are based upon the description of quandles in terms of their right translations. First, we let $(X,*)$ be a connected quandle with profile $(\ell_1,\dots,\ell_c)$, where $1=\ell_1\leq\cdots\leq \ell_c$ and $c\in \{3, 4, 5\}$. Then, we prove  that $\ell_i\mid\ell_c$, for every $i\in\{1,\dots,c\}$ by looking into the different cases. For each $c$, `$1=\ell_1\leq\cdots\leq \ell_c$' contains $c-1$ `$\leq$' signs.  Since each of these signs can take on one of two values (either `$=$' or `$<$') we break down our proof into $2^{c-1}$ distinct cases. However, we just consider the cases with $3$ or more `$<$' signs, as the remaining ones satisfy Hayashi's Conjecture in a straight-forward way, as we now show.

\begin{itemize}
    \item in the cases with $0$ `$<$' signs, $(X,*)$ has profile $(1,\dots,1)$, i.e., $(X,*)$ is a trivial quandle of order $c$, which is not even connected for $c\geq2$. There are $\binom{c-1}{0}=1$ such cases.

    \item in the cases with $1$ `$<$' sign, $(X,*)$ has profile $(1,\dots,1,\ell,\dots,\ell)$, for $1<\ell$, so it trivially satisfies Hayashi's Conjecture. There are $\binom{c-1}{1}=c-1$ such cases, according to the position of the `$<$' sign among the lengths.

    \item in the cases with $2$ `$<$' signs, $(X,*)$ has profile $(1,\dots,1,\ell,\dots,\ell, \ell',\dots,\ell')$, for $1<\ell<\ell'$. In these cases, take $P=\{1,\ell\}$ and $Q=\{\ell'\}$ in Theorem \ref{thm:generalization}. Then, either $\ell\mid\ell'$ or $\ell'\mid\ell$. As $\ell<\ell'$, we conclude that $\ell\mid\ell'$, so Hayashi's Conjecture is satisfied. There are $\binom{c-1}{2}=(c-1)(c-2)/2$ such cases, according to the position of the two `$<$' signs among the lengths.
\end{itemize}

\noindent Thus, there are $2^{c-1}-\frac{(c-1)(c-2)}{2}-c$ cases left to check. In the sequel we will prove that $\ell_i\mid\ell_c$, for every $i\in\{1,\dots,c\}$, for the $2^{c-1}-\frac{(c-1)(c-2)}{2}-c$ remaining cases, thus proving Hayashi's Conjecture for each $c\in\{3,4,5\}$. Before we move on, we remark that in the cases when there are $3$ `$<$' signs, so that $(X,*)$ has profile $(1,\dots,1,\ell,\dots,\ell, \ell',\dots,\ell',\ell'',\dots,\ell'')$, with $1<\ell<\ell'<\ell''$, we just have to prove we cannot have simultaneously $\ell\nmid\ell'$, $\ell\nmid\ell''$, $\ell'\nmid\ell''$. Indeed, suppose $$\{ 1, \ell, \ell', \ell''  \} = \{ 1, l_i, l_j, l_k \}$$and assume $$l_i\mid l_j .$$ Then set $$P=\{ 1, l_i, l_j \} \qquad \qquad Q=\{ l_k \} .$$ Then, $$l_j = \mathrm{lcm} (1, l_i, l_j) \qquad \qquad l_k = \mathrm{lcm} (l_k) $$so that, according to Theorem \ref{thm:generalization}, either $$l_j \mid l_k $$   which further implies that $l_i\mid l_j \mid l_k $ so that Hayashi's Conjecture is satisfied, or $$l_k \mid l_j$$  in which case $l_j$ is the largest element and Hayashi's Conjecture is, again, satisfied.


We write down these results in Lemma \ref{lem:summary}, below.

\begin{lemma}\label{lem:summary}
In order to prove Hayashi's Conjecture for $c\in \{ 3, 4, 5 \}$,
\begin{enumerate}[(i)]
    \item we only have to check $$2^{c-1}-\frac{(c-1)(c-2)}{2}-c $$cases, corresponding to three or more ``$<$'' signs between lengths of the profile.

\item  when there are exactly three distinct non-trivial lengths in the profile, say $ \ell_i < \ell_j < \ell_k $, we just have to prove that $$\ell_i\nmid \ell_j\qquad \qquad \text{ and } \qquad \qquad \ell_i \nmid \ell_k \qquad \qquad  \text{ and } \qquad \qquad \ell_j \nmid \ell_k $$cannot occur.
    \end{enumerate}
\end{lemma}
\subsection{Hayashi's Conjecture for $\mathbf{c=3}$}\label{subsctn:3}

Hayashi's Conjecture has already been proved for $c=3$ by Watanabe (\cite{Watanabe}), but our proof is based upon the description of quandles in terms of their right translations (see Proposition \ref{prop:c3} below).

\begin{prop}\label{prop:c3}
Hayashi's Conjecture is true for $c=3$.
\end{prop}

\begin{proof}
Let $(X,*)$ be a connected quandle with profile $(\ell_1, \ell_2, \ell_3)$, where $1=\ell_1\leq \ell_2\leq \ell_3$. We want to prove that $\ell_2\mid\ell_3$. According to Lemma \ref{lem:summary} $(i)$, there are $2^{3-1}-\frac{(3-1)(3-2)}{2}-3=0$ cases left to check, therefore the proof is complete for $c=3$.
\end{proof}

\begin{example}
The quandle $Q_{9,4}$, whose quandle table is  Table \ref{table:1}, is a connected quandle of order $9$ with profile $(1, 2, 6)$. Furthermore, $1\mid6$ and $2\mid6$, in agreement with Proposition \ref{prop:c3}.
\end{example}

\subsection{Hayashi's Conjecture for $\mathbf{c=4}$}\label{subsctn:4}

In this subsection, we prove Hayashi's Conjecture for $c=4$ using, once again, the description of quandles in terms of their right translations.

\begin{prop}\label{prop:c4}
Hayashi's Conjecture is true for $c=4$.
\end{prop}

\begin{proof}
Let $(X,*)$ be a connected quandle with profile $(\ell_1, \ell_2, \ell_3, \ell_4)$, where $1=\ell_1\leq \ell_2\leq \ell_3\leq \ell_4$. We want to prove that $\ell_2,\ell_3\mid \ell_4$. According to Lemma \ref{lem:summary}, there is $2^{4-1}-\frac{(4-1)(4-2)}{2}-4=1$ case left to check.

\vspace{5px}

\tcbox[left skip=15.5pt, sharp corners, size=small, colframe=black, colback=light-gray2]{Case\,\,\,\,\,$1=\ell_1<\ell_2<\ell_3<\ell_4$}

By Theorem \ref{thm:latin}, $(X,*)$ is latin so that left translation by $2$ is a bijection and so $\lvert C_4\rvert =\lvert \{2*i_4\,|\, i_4\in C_4\}\rvert \leq \lvert \mathcal{R}_{2,4}\rvert $. According to Lemma \ref{lem:summary} $(ii)$, suppose $\ell_2\nmid\ell_3$, $\ell_2\nmid\ell_4$, $\ell_3\nmid\ell_4$. We apply Corollary \ref{cor:latin} to find out the $C_w$'s  $\mathcal{R}_{2,4}$ is contained in. According to Corollary \ref{cor:latin}, $\ell_w\nmid \ell_2$ and  $\ell_w\nmid \ell_4$ so that $\ell_w\notin \{\ell_1, \ell_2, \ell_4\}$. So the only possible candidate this far is $\ell_w=\ell_3$. The remaining conditions for the $\ell_w$ are $(\ell_2\nmid \ell_w \lor \ell_4\mid \ell_w) \land (\ell_4\nmid \ell_w \lor \ell_2\mid \ell_w)$. So $\ell_w = \ell_3$ is the  solution. Thus, $\mathcal{R}_{2,4}\subseteq C_3$ so that $\lvert \mathcal{R}_{2,4}\rvert \leq \lvert C_3\rvert $. However, according to the standing assumption, $\lvert C_3\rvert<\lvert C_4\rvert$, which is a contradiction. The proof of Hayashi's Conjecture for $c=4$ is complete.
\end{proof}


\begin{example}
The quandle $Q_{12,4}$ (cf. \cite{Vendramin}) is a connected quandle of order $12$ with profile $(1, 2, 3, 6)$. Indeed, $1\mid6$, $2\mid6$ and $3\mid6$, in agreement with Proposition \ref{prop:c4}.
\end{example}

\subsection{Hayashi's Conjecture for $\mathbf{c=5}$}\label{subsctn:5}

In this subsection, we prove Hayashi's Conjecture for $c=5$ using, once again, the description of quandles in terms of their right translations.
\begin{prop}\label{prop:c5}
Hayashi's Conjecture  is true for $c=5$.
\end{prop}

\begin{proof}
Let $(X,*)$ be a connected quandle with profile $(\ell_1, \ell_2, \ell_3, \ell_4, \ell_5)$, where $1=\ell_1\leq \ell_2\leq \ell_3\leq \ell_4\leq \ell_5$.
By assertion \textit{1.} in Theorem \ref{thm:perm}, $\mathcal{R}_{i,1} = C_i$, for all $i\in\{1,\dots,5\}$. This will be assumed throughout the rest of the proof.


\noindent We want to prove that $\ell_2,\ell_3,\ell_4\mid \ell_5$. According to Lemma \ref{lem:summary} $(i)$, there are $2^{5-1}-\frac{(5-1)(5-2)}{2}-5=5$ cases left to check, namely
\begin{enumerate}[C{a}se 1.]
\item \qquad $1=\ell_1=\ell_2<\ell_3<\ell_4<\ell_5$;
\item \qquad $1=\ell_1<\ell_2=\ell_3<\ell_4<\ell_5$;
\item \qquad $1=\ell_1<\ell_2<\ell_3=\ell_4<\ell_5$;
\item \qquad $1=\ell_1<\ell_2<\ell_3<\ell_4=\ell_5$;
\item \qquad $1=\ell_1<\ell_2<\ell_3<\ell_4<\ell_5$.
\end{enumerate}
Cases 2., 3., and 4. correspond to cases of three distinct lengths $\ell_i<\ell_j<\ell_k$. According to Lemma \ref{lem:summary} (ii) we just have to prove that $\ell_i\nmid \ell_j, \ell_i\nmid \ell_k, \ell_j\nmid \ell_k$ cannot occur. This was done in Corollary \ref{cor:5with2equals}. Hence these Cases will not concern us anymore.

Case 1. will be proven by way of two Claims. The first Claim establishes a form for the corresponding cycle quandle Table; the second Claim proves such  Table is not consistent with the standing assumptions.

Case 5.  will be done subcase by subcase: with $i,j,k\in \{2, 3, 4\}$ we will address subcases (i) $\ell_i\nmid \ell_5$ and $\ell_j, \ell_k\mid \ell_5$; (ii) $\ell_i, \ell_j\nmid \ell_5$ and $\ell_k\mid \ell_5$; (iii) $\ell_i, \ell_j, \ell_k\nmid \ell_5$.
\vspace{5px}

\tcbox[left skip=15.5pt, sharp corners, size=small, colframe=black, colback=light-gray2]{Case 1.\qquad $1=\ell_1=\ell_2<\ell_3<\ell_4<\ell_5$}
\begin{lemma}\label{Case1}
Case 1. satisfies Hayashi's conjecture.
\end{lemma}
\begin{proof}
According to Lemma \ref{lem:summary} $(ii)$, we just have to prove that $\ell_3\nmid\ell_4$, $\ell_3\nmid\ell_5$, $\ell_4\nmid\ell_5$ cannot occur.

\begin{cl}\label{ClaimCase1Table}
Table \ref{table:4} is a cycle quandle table for Case 1. with $\ell_3\nmid\ell_4$, $\ell_3\nmid\ell_5$, $\ell_4\nmid\ell_5$.
\end{cl}
\begin{proof}
By axiom \textit{1.} in Definition \ref{def:quandle}, $\mathcal{R}_{2,2}\subseteq C_2$. Next, by Corollary \ref{cor:main}, and as $R_2$ is bijective, $\mathcal{R}_{i,2}\subseteq C_i$, for each $i\in\{1,3,4,5\}$. Besides, by Corollary \ref{cor:main}, $\mathcal{R}_{i,i}\subseteq C_{1,2,i}$, for each $i\in\{3,4,5\}$. Moreover, by Corollary \ref{cor:content}, $\mathcal{R}_{3,5},\mathcal{R}_{4,5}\subseteq C_{3,4}$, $\mathcal{R}_{3,4},\mathcal{R}_{5,4}\subseteq C_{3,5}$, $\mathcal{R}_{4,3},\mathcal{R}_{5,3}\subseteq C_{4,5}$. We leave the details of these proofs for the reader since the arguments are analogous to those used in other instances already discussed in the text.

We now prove that $\mathcal{R}_{1,i}\subseteq C_{2,i}$, for each $i\in\{3,4,5\}$. Let $i\in\{3,4,5\}$. By Corollaries \ref{cor:main} and \ref{cor:byone}, either $\mathcal{R}_{1,i}\subseteq C_1$ or $\mathcal{R}_{1,i}\subseteq C_2$ or $\mathcal{R}_{1,i}\subseteq C_i$. Suppose $\mathcal{R}_{1,i}\subseteq C_1$. Note $R_{a_i}(2)\neq2$, otherwise $R_{a_i}$ has $3$ fixed points, hence there is a $w\in C_i$ such that $R_{a_i}(w)=2$ because other options for $w$ are excluded - see entries of Table \ref{table:4} established so far. Moreover, by assertion \textit{4.} in Theorem \ref{thm:perm}, because $R_{a_i}(1)=1$, we have that $R_1 R_{a_i}= R_{a_i} R_1$. However, evaluating at $w$, while noting that $w\neq R_1(w)$, we conclude that \[2= R_1 R_{a_i}(w)= R_{a_i} R_1(w)\neq 2,\] which is a contradiction. So, $\mathcal{R}_{1,i}\subseteq C_{2,i}$, for each $i\in\{3,4,5\}$.\par

Last, we check that $\mathcal{R}_{2,i}\subseteq C_{1,i}$, for each $i\in\{3,4,5\}$. Let $i\in\{3,4,5\}$. By Corollaries \ref{cor:main} and \ref{cor:byone}, either $\mathcal{R}_{2,i}\subseteq C_1$ or $\mathcal{R}_{2,i}\subseteq C_2$ or $\mathcal{R}_{2,i}\subseteq C_i$. Suppose $\mathcal{R}_{2,i}\subseteq C_2$. Therefore, there is an $x\in C_i$ such that $R_{a_i}(x)=1$, by inspection of the entries of Table \ref{table:4} established so far. Moreover, by assertion \textit{4.} in Theorem \ref{thm:perm}, as $R_{a_i}(2)=2$, we have that $R_2 R_{a_i}= R_{a_i} R_2$. However, evaluating at $x$, while noting that $x\neq R_1(x)$, we conclude that \[1= R_1 R_{a_i}(x)= R_{a_i} R_1(x)\neq 1,\] which is a contradiction. So, $\mathcal{R}_{2,i}\subseteq C_{1,i}$, for each $i\in\{3,4,5\}$.

\begin{figure}[htbp]\centering
{\renewcommand{\arraystretch}{1.5}
\begin{tabular}{[C{24pt}"C{24pt}|C{24pt}|C{24pt}|C{24pt}|C{24pt}]}
    \thickhline
    $*$ & $C_1$ & $C_2$ & $C_3$ & $C_4$ & $C_5$\\
    \thickhline
    $C_1$ & $C_1$ & $C_1$ & $C_{2,3}$ & $C_{2,4}$ & $C_{2,5}$\\
    \hline
    $C_2$ & $C_2$ & $C_2$ & $C_{1,3}$ & $C_{1,4}$ & $C_{1,5}$\\
    \hline
    $C_3$ & $C_3$ & $C_3$ & $C_{1,2,3}$ & $C_{3,5}$ & $C_{3,4}$\\
    \hline
    $C_4$ & $C_4$ & $C_4$ & $C_{4,5}$ & $C_{1,2,4}$ & $C_{3,4}$\\
    \hline
    $C_5$ & $C_5$ & $C_5$ & $C_{4,5}$ & $C_{3,5}$ & $C_{1,2,5}$\\
    \thickhline
\end{tabular}}
\captionof{table}{Cycle quandle table for $(X,*)$ when $1=\ell_1=\ell_2<\ell_3<\ell_4<\ell_5$ and $\ell_3\nmid \ell_4$, $\ell_3\nmid \ell_5$, $\ell_4\nmid \ell_5$.}\label{table:4}
\end{figure}
\end{proof}
\begin{cl}\label{ClaimCase1Clash}
Table \ref{table:4} is not compatible with the standing assumptions.
\end{cl}
\begin{proof}
We now prove that if $\mathcal{R}_{1,i}\subseteq C_2$, then $\mathcal{R}_{2,i}\subseteq C_1$, for each $i\in\{3,4,5\}$. In order to reach a contradiction, suppose $\mathcal{R}_{1,i}\subseteq C_2$ and assume $\mathcal{R}_{2,i}\nsubseteq C_1$. Then, there is $y\in C_i$ such that $R_{a_i}(y)=1$. By assertion \textit{3.} in Theorem \ref{thm:perm}, as $R_{a_i}(1)\in C_2$, we have that $R_1 R_2 R_1^{-1}= R_{a_i} R_1 R_{a_i}^{-1}$. However, evaluating at $1$, while noting that $y\neq R_1(y)$, we get \[1= R_1 R_2 R_1^{-1}(1)= R_{a_i} R_1 R_{a_i}^{-1}(1)\neq 1,\] which is a contradiction. Hence if $\mathcal{R}_{1,i}\subseteq C_2$, then $\mathcal{R}_{2,i}\subseteq C_1$, for each $i\in\{3,4,5\}$. We further prove that if $\mathcal{R}_{2,i}\subseteq C_1$, then $\mathcal{R}_{1,i}\subseteq C_2$, for each $i\in\{3,4,5\}$. Suppose $\mathcal{R}_{2,i}\subseteq C_1$ and let $z\in C_i$ be such that $R_{a_i}(z)=2$ (instead of $R_{a_i}(1)=2$). By assertion \textit{5.} in Theorem \ref{thm:perm}, as $R_{a_i}(2)=1$, we have $R_{a_i}^{-1} R_1 R_{a_i}= R_1 R_2 R_1^{-1}$. But evaluating at $z$, we get \[z= R_{a_i}^{-1} R_1 R_{a_i}(z)= R_1 R_2 R_1^{-1}(z)\Leftrightarrow R_2 R_1^{-1}(z)= R_1^{-1}(z).\] Since $z\in C_i$, $R_2$ has $3$ fixed points, which is a contradiction. Hence, if $\mathcal{R}_{2,i}\subseteq C_1$, then $\mathcal{R}_{1,i}\subseteq C_2$, for each $i\in\{3,4,5\}$. Thus, we conclude that either $\mathcal{R}_{1,i}\subseteq C_2$ and $\mathcal{R}_{2,i}\subseteq C_1$ or $\mathcal{R}_{1,i},\mathcal{R}_{2,i}\subseteq C_i$, for each $i\in\{3,4,5\}$.\par

Now, on one hand, if $\mathcal{R}_{1,i}\subseteq C_2$ and $\mathcal{R}_{2,i}\subseteq C_1$, for a certain $i\in\{3,4,5\}$, then $(C_{1,2,i},*)$ is a subquandle of $(X,*)$ with profile $(1,1,\ell_i)$, where $\ell_i=2$ (since, for any $u_i\in C_i$,  $(1*u_i)*u_i=2*u_i=1$). On the other hand, if $\mathcal{R}_{1,i},\mathcal{R}_{2,i}\subseteq C_i$, for a certain $i\in\{3,4,5\}$, then $(C_{1,2,i},*)$ is a subquandle of $(X,*)$ (by inspection of Table \ref{table:4}) such that $$R_{a_i}(1)=b_i\neq c_i=R_{a_i}(2) \qquad R_{a_i}(a_i)=a_i\neq d_i=R_{a_i}(d_i) \,\, (\text{fixed points})$$ so $(C_{1,2,i},*)$ has at least six elements: $1, 2, a_i, d_i, b_i, c_i$. According to \cite{Lages_Lopes}, quandles whose profile is of the sort $(1,\dots ,  1, \ell)$ and whose order exceeds twice the number of fixed points are connected - which is precisely what happens in the case at hand: order greater or equal to six, two fixed points. Furthermore, there is only one such connected quandle and its profile is $(1,1,4)$ (again according to \cite{Lages_Lopes}). Thus, the profile of $(C_{1,2,i},*)$ is $(1, 1, 4)$. So $\ell_i$ can only take on two values, either $2$ (when $\mathcal{R}_{1,i}\subseteq C_2$ and $\mathcal{R}_{2,i}\subseteq C_1$) or $4$ (when $\mathcal{R}_{1,i},\mathcal{R}_{2,i}\subseteq C_i$). But $\ell_i$ represents one of three lengths which are pairwise distinct by the standing assumptions. This concludes the proof of Claim \ref{ClaimCase1Clash}.
\end{proof}
This concludes the Proof of Lemma \ref{Case1}.
\end{proof}
\vspace{5px}

\tcbox[left skip=15.5pt, sharp corners, size=small, colframe=black, colback=light-gray2]{Case 5.\qquad $1=\ell_1<\ell_2<\ell_3<\ell_4<\ell_5$}
\begin{lemma}\label{Case5}
Case 5. satisfies Hayashi's conjecture.
\end{lemma}
\begin{proof}
By Theorem \ref{thm:latin}, $(X,*)$ is latin. We want to prove that $\ell_2,\ell_3,\ell_4\mid \ell_5$. Let $\{i,j,k\}=\{2,3,4\}$. There are three cases to be checked: $(i)\, \ell_i\nmid \ell_5$; $(ii)\, \ell_i,\ell_j\nmid \ell_5$; $(iii)\, \ell_i,\ell_j,\ell_k\nmid \ell_5$.\par

\begin{enumerate}[(i)]
\item Suppose $\ell_i\nmid \ell_5$ and $\ell_j,\ell_k\mid \ell_5$. Then, by Corollary \ref{cor:latin}, we have $\mathcal{R}_{i,5}\subseteq\varnothing$, which is a contradiction.\par

\item Suppose $\ell_i,\ell_j\nmid \ell_5$ and $\ell_k\mid \ell_5$. Then, using Corollary \ref{cor:latin}, $\mathcal{R}_{i,5}\subseteq C_j$. However, because $\lvert C_j\rvert<\lvert C_5\rvert$, $(X,*)$ cannot be latin, which is a contradiction.

\item Suppose $\ell_i,\ell_j,\ell_k\nmid \ell_5$. Then Table \ref{table:8} is a cycle quandle table for $(X,*)$. Indeed, by Corollary \ref{cor:main}, and as each permutation of $(X,*)$ has just one fixed point, $\mathcal{R}_{1,5}\subseteq C_5$. Besides, by Corollary \ref{cor:main}, $\mathcal{R}_{5,5}\subseteq C_{1,5}$. Furthermore, using Corollary \ref{cor:latin}, $\mathcal{R}_{2,5}\subseteq C_{3,4}$, $\mathcal{R}_{3,5}\subseteq C_{2,4}$, $\mathcal{R}_{4,5}\subseteq C_{2,3}$.

\begin{figure}[htbp]\centering
{\renewcommand{\arraystretch}{1.5}
\begin{tabular}{[C{24pt}"C{24pt}|C{24pt}|C{24pt}|C{24pt}|C{24pt}]}
    \thickhline
    $*$ & $C_1$ & $C_2$ & $C_3$ & $C_4$ & $C_5$\\
    \thickhline
    $C_1$ & $C_1$ &  &  &  & $C_5$\\
    \hline
    $C_2$ & $C_2$ &  &  &  & $C_{3,4}$\\
    \hline
    $C_3$ & $C_3$ &  &  &  & $C_{2,4}$\\
    \hline
    $C_4$ & $C_4$ &  &  &  & $C_{2,3}$\\
    \hline
    $C_5$ & $C_5$ &  &  &  & $C_{1,5}$\\
    \thickhline
\end{tabular}}
\captionof{table}{Cycle quandle table for $(X,*)$ when $1=\ell_1<\ell_2<\ell_3<\ell_4<\ell_5$ and $\ell_2\nmid \ell_5$, $\ell_3\nmid \ell_5$, $\ell_4\nmid \ell_5$.}\label{table:8}
\end{figure}

Now, by assertion \textit{3.} in Theorem \ref{thm:perm}, as $R_{a_5}(1)\in C_5$, we have $R_1^{ R_{a_5}(1)-a_4} R_{a_5} R_1^{-( R_{a_5}(1)-a_4)}= R_{a_5} R_1 R_{a_5}^{-1}$. Evaluating at $z\in C_4$, we conclude that \[C_{2,3}\ni R_1^{R_{a_5}(1)-a_4} R_{a_5} R_1^{-( R_{a_5}(1)-a_4)}(z)= R_{a_5} R_1 R_{a_5}^{-1}(z).\] Notice that $R_1 R_{a_5}^{-1}(z)\in C_{2,3}$. So, for each $z\in C_4$, we have that $R_{a_5}(z),R_{a_5}(R_1 R_{a_5}^{-1}(z))\in C_{2,3}$. In particular, as $R_{a_5}$ is bijective and $R_1 R_{a_5}^{-1}(z)\in C_{2,3}$ then,$$R_{a_5}(C_4)\cap R_{a_5}(R_1 R_{a_5}^{-1}(C_4))= \emptyset .$$ Therefore, $$|C_4|+|C_4|=|R_{a_5}(C_4)\cup R_{a_5}(R_1 R_{a_5}^{-1}(C_4))|\leq |C_{2,3}|$$ that is, $\lvert C_{2,3}\rvert\geq 2\cdot\lvert C_4\rvert$, which conflicts with the standing assumptions.
\end{enumerate}
This concludes the proof of Lemma \ref{Case5}.
\end{proof}
The proof of Hayashi's Conjecture for $c=5$ is complete concluding the proof of Proposition \ref{prop:c5}.
\end{proof}

\begin{example}
The quandle $Q_{15,3}$ (cf. \cite{Vendramin}) is a connected quandle of order $15$ with profile $(1, 2, 4, 4, 4)$. Furthermore, $1\mid4$, $2\mid4$ and $4\mid4$, in agreement with Proposition \ref{prop:c5}.
\end{example}

\section{Final Remarks}\label{sctn:further}

The approach  used here to prove Hayashi's Conjecture for $c\in\{3,4,5\}$ might give ideas to prove the same conjecture for greater values of $c$. Hopefully, more general arguments will be developed so that computer resources would be useful. Alternatively, the same approach could  be used to identify families of quandles for which  this conjecture does not apply.

\end{document}